\documentclass[11pt,reqno]{amsart}
\usepackage{amsthm,amsfonts,amsmath,amssymb}
\usepackage[pagebackref=true,backref=true,bookmarksopen=true]{hyperref}
\usepackage[letterpaper,margin=1in]{geometry}

\newtheorem{theorem}{Theorem}
\newtheorem{cnj}{Conjecture}
\newtheorem{lemma}{Lemma}

\newcommand{\R}{{\mathbb R}}

\begin{document}

\title{Separation for Schr\"odinger operators. A counterexample to Simon's conjecture}
\author{M.A. Perelmuter}
\address{SCAD Soft Ltd., 3a Osvity Street, Kyiv, 03037, Ukraine}
\email{mikeperelmuter@gmail.com}

\begin{abstract}
We construct, for every $d\ge 5$, a nonnegative radial potential
$V\in L^\infty_{\mathrm{loc}}(\mathbb{R}^d\setminus\{0\})\cap L^1(\mathbb{R}^d)$ for which
$V(-\Delta\dotplus V)^{-1}$ is not bounded on $L^2(\mathbb{R}^d)$.
\end{abstract}

\maketitle
\section{Introduction}
Let $H_0=-\Delta$ on $L^2(\R^d)$. $H=H_0\dotplus V$ is defined as a self-adjoint operator via the form sum.
Clearly, $\mathcal{D}(H_0)\cap \mathcal{D}(V) \subset \mathcal{D}(H)$ and any condition implying
\begin{equation}\label{Separate}
\mathcal{D}(H) \subset \mathcal{D}(V)
\end{equation}
implies $H=H_0+V$. W.~N.~Everitt and M.~Giertz, in \cite{Ever}, coined the term {\it separation property} for relation \eqref{Separate}.
There are many publications devoted to sufficient conditions for the separation property of various differential operators. We mention only a few of them here: \cite{Ever}, \cite{Ever73}, \cite{Ever77}, \cite{Davies}, \cite{Davies2}, \cite{Otelbaev}, \cite{Kalf}, \cite{Okazawa},  \cite{GS}, \cite{MS}.

In 1983 B.~Simon \cite{Simon2} formulated the following conjecture
\begin{cnj}\label{SimonCnj}
Let $V(x)=V(|x|)$ on $\R^d; d\geq 5$. Suppose
\begin{enumerate}
  \item $V\geq 0$
  \item $V$ is bounded on any closed set avoiding $0$.
\end{enumerate}
Then $H_0+V$ is closed on $\mathcal{D}(H_0)\cap \mathcal{D}(V)$.
\end{cnj}
The fact that the conjecture is false for $d=4$ was proved by B.~Simon \cite{Simon2} himself. To the best of our knowledge, this conjecture has remained open for $d\geq 5$, and the purpose of this note is to construct a counterexample.
\par
We construct a radial potential $V\ge0$, singular only at the origin, for which
$V(-\Delta+V)^{-1}$ fails to be bounded on $L^2(\R^d)$. The potential
is supported on a countable family of thin, pairwise disjoint spherical shells
$A_n=[r_n,\,r_n+w_n],\ r_n \to 0$, accumulating at the origin, on which $V$ takes the value
$h_n\sim r_n^{\,1-d}$. The widths $w_n$ are chosen so that the quantity $h_n w_n$, which may be viewed as the mass of each shell, is independent of $n$: the shells become thinner and taller as $n\to\infty$.
\par
For this potential we construct $f\in L^2(\R^d)$ and $u\in \mathcal{D}(H)$ such that $Hu=f$ and $Vu\notin L^2(\R^d)$. By spherical symmetry the construction reduces to a one-dimensional Bessel-type equation, which is then converted into a Volterra integral equation with a special integral kernel.

\section{A counterexample}
\begin{theorem}
Let $d\ge 5$. There exists a radial potential
$$0\le V\in L^\infty_{\mathrm{loc}}(\R^d\setminus\{0\})\cap L^1({\R^d})$$
such that the operator $V(-\Delta\dotplus V)^{-1}$ is not bounded on $L^2(\R^d)$.
\end{theorem}
\begin{proof}
Let $r_n=2^{-n},\ w_n=r_n^{d-1},\ n\ge 2,$ and $A_n=[r_n,r_n+w_n]$.
Since $\frac{w_n}{r_n}=r_n^{d-2}<1,$ we have $r_n+w_n<2r_n=r_{n-1},$ so the intervals $A_n$ are pairwise disjoint.
\par
Fix $\lambda>0$ and define $h_n=\frac{\lambda}{w_n}=\lambda r_n^{1-d}$,
$$V(r)=
\begin{cases}
h_n,&r\in A_n,\\
0,&r\notin\bigcup_{n\ge 2}A_n.
\end{cases}$$
Thus $V$ is radial and nonnegative. Since every compact subset of $\R^d\setminus\{0\}$ intersects only finitely many layers, we have $V\in L^\infty_{\mathrm{loc}}(\R^d\setminus\{0\}).$ The inclusion $V\in L^1(\R^d)$ is a consequence of the equality $V(r)=0$ for $r\notin\bigcup_{n\ge 2}A_n$ and the estimate
\begin{equation*}
\|V\|_{L^1(\R^d)}=\omega_{d-1}\int\limits_0^1 r^{d-1}V(r)\,\mathrm{d}r\le \omega_{d-1} 2^{d-1}\lambda\sum_{n\ge 2}r_n^{d-1}<\infty,
\end{equation*}
where $\omega_{d-1}$ is the surface area of the unit $(d-1)$-dimensional sphere.
\par
Define the quadratic form
$$q[u]=\int\limits_{\R^d}|\nabla u|^2\,\mathrm{d}x+\int\limits_{\R^d}V|u|^2\,\mathrm{d}x$$
 with domain $\mathcal{D}(q)=H^1(\R^d)\cap L^2(\R^d,V\,\mathrm{d}x)$. Since $0\le V\in L^1(\R^d)$, this is a densely defined closed nonnegative quadratic form. Let $H$ be the associated self-adjoint operator (the form realization of $-\Delta+V$).
 Let $f(x)=\chi_{\{1\le |x|\le 2\}}\in L^2(\R^d)$ (the indicator function of the set $\{1\le |x|\le 2\}$). We shall construct a radial solution $u\in \mathcal{D}(H)$ such that $Hu=f$. It will then follow that $u=H^{-1}f=(-\Delta\dotplus V)^{-1}f.$
\par
To construct the function $u$, we will consider three intervals: $(0,1), [1,2], (2,\infty)$.
\par
1. As a first step we consider a solution near the origin ($0 < r < 1$).
By spherical symmetry, the construction of the function $u$ can be reduced to solving a one-dimensional Bessel-type equation.
Consider the radial homogeneous equation
$$-U''(r)-\frac{d-1}{r}U'(r)+V(r)U(r)=0.$$
Define $U$ as the solution of the Volterra equation
\begin{equation}\label{UVoltera}
U(r)=1+\frac{1}{d-2}\int\limits_0^r\left(s-r^{2-d}s^{d-1}\right)V(s)U(s)\,\mathrm{d}s.
\end{equation}
Since $V\ge 0$, the kernel is nonnegative, and $0\le s-r^{2-d}s^{d-1}\le s,\  0<s<r$. Moreover,
\begin{equation}\label{rV}
\int\limits_0^1 rV(r)\,\mathrm{d}r=\sum_{n\ge 2}h_n\int\limits_{r_n}^{r_n+w_n}r\,\mathrm{d}r\le\sum_{n\ge 2}h_n(r_n+w_n)w_n\le
2\lambda\sum_{n\ge 2}r_n<\infty.
\end{equation}
By the standard existence and uniqueness theory for linear Volterra integral equations of the second kind with integrable kernels, equation \eqref{UVoltera} has a unique continuous solution on $[0,1]$.
\par
Since the kernel is nonnegative, successive approximations starting with $U_0=1$ show that
\begin{equation}\label{Ule1}
U(r)\ge 1,\quad r\in[0,1].
\end{equation}
Differentiating \eqref{UVoltera}, we obtain
$$U'(r)=r^{1-d}\int\limits_0^r s^{d-1}V(s)U(s)\,\mathrm{d}s.$$
Since $U\ge 0$, we get $U'\ge 0$.
The formula for $U'$ also shows that $U\in C^1((0,1])$ and that $U'$ is locally absolutely continuous on $(0,1]$.
Moreover, \eqref{UVoltera} gives
$$U(r)\le 1+\frac{1}{d-2}\int\limits_0^r sV(s)U(s)\,\mathrm{d}s.$$
Hence Gronwall's inequality and \eqref{rV} imply
\begin{equation}\label{Uge1}
1\le U(r)\le\exp\left(\frac{1}{d-2}\int\limits_0^1sV(s)\,\mathrm{d}s\right)=:C_0<\infty.
\end{equation}
We next record an estimate for $U'$ valid on all of $(0,1)$, not merely on the layers $A_n$ themselves. For $r\in(0,1)$ let $n=n(r)\ge 2$ be the unique integer with $r\in[r_n,r_{n-1})$, where $r_1:=1$.
Since $A_k\subset[r_k,r_{k-1})$ for every $k\ge 2$, the set $\{s\le r:V(s)\ne 0\}$ is contained in $\bigcup_{k\ge n}A_k$, and $r^{1-d}\le r_n^{1-d}$ because $r\ge r_n$. Hence
\begin{equation}\label{UPrime}
\begin{aligned}
U'(r)
&\le
C_0r^{1-d}\int\limits_0^r s^{d-1}V(s)\,\mathrm{d}s\\
&\le
C_0 r_n^{1-d}
\sum_{k\ge n}h_kw_k (r_k+w_k)^{d-1}\\
&\le
C\lambda r_n^{1-d}\sum_{k\ge n}r_k^{d-1}
\le C.
\end{aligned}
\end{equation}
The last inequality follows from
$$r_n^{1-d}\sum_{k\ge n}r_k^{d-1}=\sum_{j\ge 0}2^{-j(d-1)}=\frac{1}{1-2^{1-d}}\le 2.$$
Here and below $C$ denotes a finite constant independent of $n$ (hence of $r$). Since this bound holds for every $r\in(0,1)$, we conclude
$$\int\limits_0^1|U'(r)|^2r^{d-1}\,\mathrm{d}r\le C^2\int\limits_0^1r^{d-1}\,\mathrm{d}r<\infty.$$
\par\medskip
On $(0,1)$ put
\begin{equation}\label{ucU}
u(r)=cU(r),
\end{equation}
where $c>0$ will be chosen below so that the solutions can be matched continuously with their first derivatives at the interval endpoints.
\par\medskip
2. On $1\le r \le 2$, $V=0$ and $f=1$. Using the standard representation of the Laplace operator acting on radial functions, we have
\begin{equation}\label{r12}
-\left(r^{d-1}u'(r)\right)'=r^{d-1}.
\end{equation}
With the initial conditions $u(1)=cU(1)$ and $u'(1)=cU'(1)$, there exists a unique solution on $[1,2]$.
\par\medskip
3. For $r>2$, we have $V=f=0$.  We consider the homogeneous equation
$$-\bigl(r^{d-1}v'(r)\bigr)'=0$$
and choose the decaying solution
$$v(r)=b(c)r^{2-d},$$
where $b(c)$ is chosen so that the solution is continuous at $r=2$, namely, $b(c)=2^{d-2}u(2)$. For this choice, we have $v'(2)=-\frac{d-2}{2}u(2)$.
Consequently, in order for the solution defined on $[1,2]$ to be $C^1$-matched with the exterior solution at $r=2$, we need $u(2)+\frac{2}{d-2}u'(2)=0$.
\par
\par
By Lemma~\ref{lem:gluing} below, the unique solution of \eqref{r12} on $[1,2]$ satisfying $u(1)=cU(1)$, $u'(1)=cU'(1)$ satisfies
$$u(2)+\frac{2}{d-2}u'(2)=c\left(U(1)+\frac{U'(1)}{d-2}\right)-\frac{3}{2(d-2)}.$$
Since $U(1)\ge 1$ and $U'(1)\ge 0$, the unique value of $c$ satisfying $u(2)+\frac{2}{d-2}u'(2)=0$ is
$$c=\frac{3}{2\bigl((d-2)U(1)+U'(1)\bigr)}>0.$$
For this choice of $c$, defining
$$u(r)=b(c)r^{2-d},\qquad r\geq 2,$$
gives a $C^1$ solution across $r=2$.
\par\medskip
Since $d\ge 5$,
$$\int\limits_2^\infty |u(r)|^2r^{d-1}\,\mathrm{d}r=b^2\int\limits_2^\infty r^{3-d}\,\mathrm{d}r<\infty.$$
Together with \eqref{Uge1}, this gives $u\in L^2(\R^d)$.
\par\medskip
Now our function $u$ is defined for all $r\in (0,\infty)$ and we have to verify that $u$ belongs to $\mathcal{D}(H)$. We first prove that $u\in H^1(\R^d)$. Near the origin, \eqref{UPrime} implies
$$\int\limits_0^1|U'(r)|^2r^{d-1}\,\mathrm{d}r<\infty.$$
Since $u=cU$ on $(0,1)$, the same integrability property holds for $u$.
On $(1,2)$ the solution is smooth, and for $r\ge 2$ the formula $u(r)=br^{2-d}$ gives $|u'(r)|\lesssim r^{1-d}$. Hence
$$\int\limits_2^\infty|u'(r)|^2r^{d-1}\,\mathrm{d}r<\infty.$$
Thus $u\in H^1(\R^d)$.
\par
Next we estimate the potential part of the quadratic form. By \eqref{Uge1}, $u(r)\le cC_0,\  0<r\le 1$.
Therefore, using $V\in L^1(\R^d)$ and $V=0$ on $1\le r$, we have
$$\int\limits_{\R^d}V|u|^2\,\mathrm{d}x =\int\limits_{|x|<1}V|u|^2\,\mathrm{d}x\le c^2C_0^2\|V\|_{L^1(\R^d)}<\infty.$$
Thus $u\in \mathcal{D}(q)$, which yields a global radial solution of
$-\Delta u+Vu=f$. This differential equation holds classically away from the layer endpoints, and $u$ and $u'$ are continuous across all these endpoints. By \eqref{Uge1} and \eqref{UPrime}, $u$ and $u'$ are bounded near the origin. Hence both boundary terms arising from integration by parts over $\{|x|>\varepsilon\}$ tend to zero as $\varepsilon\to 0$. Thus it holds in the distributional sense on $\R^d$:
\begin{equation}\label{Distributional}
-\Delta u+Vu=f.
\end{equation}
Here $Vu\in L^1_{\mathrm{loc}}(\R^d)$ because $u$ is locally bounded and $V\in L^1(\R^d)$.
\par
Hence, for every $\varphi\in C_c^\infty(\R^d)$,
$$q[u,\varphi]=\int\limits_{\R^d}f\overline{\varphi}\,\mathrm{d}x.$$
Since $C_c^\infty(\mathbb R^d)$ is a form core for $q$ (see, for example, \cite[Theorem 2.1]{Simon1}), this identity extends to every $\varphi\in\mathcal D(q)$. By the definition of the operator associated with the closed quadratic form $q$, it follows that
$$u\in\mathcal D(H),\quad Hu=f.$$
\par\medskip
To complete the proof of the theorem, it remains to show that $Vu\notin L^2(\R^d)$.
\par
By \eqref{Ule1} and \eqref{ucU}, $u(r)\ge c,\  0<r<1$.  We use the same notation $A_n$ for the spherical shell $\{x\in\R^d: |x|\in A_n\}$. On $A_n$ we have $V=h_n$. Hence
$$\int\limits_{A_n}(Vu)^2\,\mathrm{d}x = \omega_{d-1}h_n^2 \int\limits_{r_n}^{r_n+w_n}u(r)^2r^{d-1}\,\mathrm{d}r\ge c^2\omega_{d-1}h_n^2r_n^{d-1}w_n.$$
Since $h_n=\frac{\lambda}{w_n},\  w_n=r_n^{d-1}$, we obtain $h_n^2r_n^{d-1}w_n=\lambda^2$.
Consequently
$$\int\limits_{A_n}(Vu)^2\,\mathrm{d}x\ge c^2\omega_{d-1}\lambda^2=:\delta>0,$$
where $\delta$ is $n$-independent.
\par
The sets $A_n$ are pairwise disjoint, and therefore
$$\|Vu\|_{L^2(\R^d)}^2\ge\sum_{n\ge 2}\int\limits_{A_n}(Vu)^2\,\mathrm{d}x\ge\sum_{n\ge 2}\delta=\infty.$$
Thus $Vu\notin L^2(\R^d)$. Hence $V(-\Delta\dotplus V)^{-1}$ is not a bounded operator on $L^2(\R^d)$.
\end{proof}

\noindent\emph{Remark.} The condition $d>4$ enters through the square integrability of the exterior solution $r^{2-d}$ at infinity. For $d=4$, the integral $\int\limits_2^\infty r^{3-d}\,\mathrm{d}r$ diverges logarithmically. As mentioned in the introduction, a counterexample in dimension four was constructed by Simon \cite{Simon2}.

\section{Appendix: the gluing identity}
\begin{lemma}\label{lem:gluing}
Let $d\ge 3$, and let $u$ be the unique solution of
$$-\left(r^{d-1}u'(r)\right)'=r^{d-1},\qquad 1\le r\le 2,$$
subject to the initial conditions $u(1)=cU(1)$, $u'(1)=cU'(1)$, where $c,\,U(1),\,U'(1)$ are given real constants. Then
$$u(2)+\frac{2}{d-2}u'(2)=c\left(U(1)+\frac{U'(1)}{d-2}\right)-\frac{3}{2(d-2)}.$$
\end{lemma}
\begin{proof}

Writing the equation as $(r^{d-1}u')'=-r^{d-1}$, a particular solution is sought in the form $u_p(r)=\alpha r^2$. Since
$$u_p'(r)=2\alpha r,\qquad r^{d-1}u_p'(r)=2\alpha r^{d},\qquad \left(r^{d-1}u_p'\right)'(r)=2\alpha d\,r^{d-1},$$
matching $2\alpha d=-1$ gives $\alpha=-\dfrac{1}{2d}$, so $u_p(r)=-\dfrac{r^2}{2d}$.

The homogeneous equation $(r^{d-1}u')'=0$ gives $r^{d-1}u'=\text{const}$, hence $u'=Cr^{1-d}$ and $u=A+Br^{2-d}$. The general solution of the inhomogeneous equation is therefore
$$u(r)=A+Br^{2-d}-\frac{r^2}{2d},\qquad u'(r)=B(2-d)r^{1-d}-\frac{r}{d}.$$

Imposing the initial conditions at $r=1$:
$$u(1)=A+B-\frac{1}{2d}=cU(1),\qquad u'(1)=B(2-d)-\frac1d=cU'(1).$$
The second equation gives
$$B=-\frac{cU'(1)+\frac1d}{d-2}=-\frac{cU'(1)}{d-2}-\frac{1}{d(d-2)},$$
and the first gives
$$A=cU(1)-B+\frac{1}{2d}.$$

Now evaluate at $r=2$:
$$u(2)=A+B\,2^{2-d}-\frac2d,\qquad u'(2)=B(2-d)2^{1-d}-\frac2d.$$
Since $\dfrac{2-d}{d-2}=-1$,
$$\frac{2}{d-2}u'(2)=\frac{2}{d-2}\Big[B(2-d)2^{1-d}-\frac2d\Big]=-B\,2^{2-d}-\frac{4}{d(d-2)}.$$
Adding this to $u(2)$, the terms containing $B\,2^{2-d}$ cancel:
$$u(2)+\frac{2}{d-2}u'(2)=A-\frac2d-\frac{4}{d(d-2)}.$$

Substituting $A=cU(1)-B+\dfrac{1}{2d}$ and then $B=-\dfrac{cU'(1)}{d-2}-\dfrac{1}{d(d-2)}$:
\begin{align*}
u(2)+\frac{2}{d-2}u'(2)
&=cU(1)-B+\frac{1}{2d}-\frac2d-\frac{4}{d(d-2)}\\
&=cU(1)-B-\frac{3}{2d}-\frac{4}{d(d-2)}\\
&=cU(1)+\frac{cU'(1)}{d-2}+\frac{1}{d(d-2)}
-\frac{3}{2d}-\frac{4}{d(d-2)}\\
&=c\left(U(1)+\frac{U'(1)}{d-2}\right)
-\frac{3}{d(d-2)}-\frac{3}{2d}.
\end{align*}

Finally, the remaining constants combine as
$$\frac{3}{d(d-2)}+\frac{3}{2d}=\frac3d\left(\frac{1}{d-2}+\frac12\right)=\frac3d\cdot\frac{2+(d-2)}{2(d-2)}=\frac3d\cdot\frac{d}{2(d-2)}=\frac{3}{2(d-2)},$$
which yields the claimed identity.

\end{proof}

\section*{Acknowledgments}

AI tools were used to assist with language editing
and verification of routine mathematical calculations.
The author has reviewed and verified the entire manuscript
and takes full responsibility for its content and
mathematical accuracy.

\end{document}